\documentclass[sts]{imsart}
\usepackage{algorithm}
\usepackage{algorithmic}

\usepackage{amsmath, amssymb, color}
\usepackage{amsthm} 

\usepackage[colorlinks,citecolor=blue,urlcolor=blue]{hyperref}
\usepackage{xcolor,colortbl}
\RequirePackage{hypernat}

\usepackage{color}
\usepackage{amsfonts, fancybox}

\usepackage[american]{babel}
\usepackage{multirow}
\usepackage{color}
\usepackage{amsmath}
\usepackage{float}
\usepackage{mathrsfs}
\usepackage{amsfonts}
\usepackage{dsfont}
\usepackage{amssymb}

\usepackage{graphicx}
\usepackage{algorithmic, algorithm}
\usepackage{wrapfig}

\usepackage{amssymb, latexsym}

\usepackage{enumerate}

\usepackage{subfig}
\usepackage{dsfont}

\startlocaldefs

\newtheorem{lemma}{Lemma}
\newtheorem{definition}{Definition}
\newtheorem{proposition}{Proposition}

\newtheorem{theorem}{Theorem}

\newtheorem{assumption}{Assumption}

\newcommand{\bc}{\begin{center}}
\newcommand{\ec}{\end{center}}
\newcommand{\ba}{\begin{array}}
\newcommand{\ea}{\end{array}}
\newcommand{\be}{\begin{eqnarray}}
\newcommand{\ee}{\end{eqnarray}}
\newcommand{\bel}{\begin{eqnarray}\label}
\newcommand{\eel}{\end{eqnarray}}
\newcommand{\bes}{\begin{eqnarray*}}
\newcommand{\ees}{\end{eqnarray*}}
\newcommand{\bn}{\begin{enumerate}}
\newcommand{\en}{\end{enumerate}}

\newcommand{\rP}{{\mathbb{P}}}

\newcommand{\DPP}{\mathsf{DPP}}
\newcommand{\cD}{\mathcal{D}}
\newcommand{\cC}{\mathcal{C}}
\newcommand{\cI}{\mathcal{I}}

\newcommand{\etal}{{\em et al. }}

\newcommand{\sgn}{\operatorname{sgn}}

\newcommand{\R}{{\rm I}\kern-0.18em{\rm R}}
\newcommand{\h}{{\rm I}\kern-0.18em{\rm H}}
\newcommand{\K}{{\rm I}\kern-0.18em{\rm K}}
\newcommand{\p}{{\rm I}\kern-0.18em{\rm P}}
\newcommand{\E}{{\rm I}\kern-0.18em{\rm E}}

\newcommand{\1}{{\rm 1}\kern-0.24em{\rm I}}
\newcommand{\N}{{\rm I}\kern-0.18em{\rm N}}

\definecolor{MIT}{RGB}{163,31,52}
\endlocaldefs

\begin{document}

\begin{frontmatter}

\title{Learning Determinantal Point Processes with Moments and Cycles}
\runtitle{Learning DPPs with Moments and Cycles}

\begin{aug}

\author{\fnms{John}~\snm{Urschel}\ead[label=urschel]{urschel@mit.edu}},
\author{\fnms{Victor-Emmanuel}~\snm{Brunel}\ead[label=veb]{veb@mit.edu}},
\author{\fnms{Ankur}~\snm{Moitra}\thanksref{t3}\ead[label=moitra]{moitra@mit.edu}}
\and
\author{\fnms{Philippe}~\snm{Rigollet}\thanksref{t2}\ead[label=rigollet]{rigollet@math.mit.edu}},

\affiliation{Massachusetts Institute of Technology}

\thankstext{t2}{This work was supported in part by NSF CAREER DMS-1541099, NSF DMS-1541100, DARPA W911NF-16-1-0551, ONR N00014-17-1-2147 and a grant from the MIT NEC Corporation.}
\thankstext{t3}{This work was supported in part by NSF CAREER Award CCF-1453261, NSF Large CCF-1565235, a David and Lucile Packard Fellowship, an Alfred P. Sloan Fellowship, an Edmund F. Kelley Research Award, a Google Research Award and a grant from the MIT NEC Corporation.}

%
%

\address{{John Urschel}\\
{Department of Mathematics} \\
{Massachusetts Institute of Technology}\\
{77 Massachusetts Avenue,}\\
{Cambridge, MA 02139-4307, USA}\\
\printead{urschel}
}

\address{{Victor-Emmanuel Brunel}\\
{Department of Mathematics} \\
{Massachusetts Institute of Technology}\\
{77 Massachusetts Avenue,}\\
{Cambridge, MA 02139-4307, USA}\\
\printead{veb}
}

\address{{Ankur Moitra}\\
{Department of Mathematics} \\
{Massachusetts Institute of Technology}\\
{77 Massachusetts Avenue,}\\
{Cambridge, MA 02139-4307, USA}\\
\printead{moitra}
}

\address{{Philippe Rigollet}\\
{Department of Mathematics} \\
{Massachusetts Institute of Technology}\\
{77 Massachusetts Avenue,}\\
{Cambridge, MA 02139-4307, USA}\\
\printead{rigollet}
}

\runauthor{Urschel et al.}
\end{aug}

\begin{abstract}
Determinantal Point Processes (DPPs) are a family of probabilistic models that have a repulsive behavior, and lend themselves naturally to many tasks in machine learning where returning a diverse set of objects is important. While there are fast algorithms for sampling, marginalization and conditioning, much less is known about learning the parameters of a DPP. Our contribution is twofold: (i) we establish the optimal sample complexity achievable in this problem and show that it is governed by a natural parameter, which we call the \emph{cycle sparsity}; (ii) we propose a provably fast combinatorial algorithm that implements the method of moments efficiently and achieves optimal sample complexity. Finally, we give experimental results that confirm our theoretical findings. 
\end{abstract}


\end{frontmatter}

\section{Introduction}
\label{SEC:intro}

Determinantal Point Processes (DPPs) are a family of probabilistic models that arose from the study of quantum mechanics \cite{Mac75} and random matrix theory \cite{Dys62}. Following the seminal work of Kulesza and Taskar \cite{KulTas12}, discrete DPPs have found numerous applications in machine learning, including in document and timeline summarization \cite{LinBil12, YaoFanZha16}, image search \cite{KulTas11,AffFoxAda14} and segmentation \cite{LeeChaYan16}, audio signal processing \cite{XuOu16}, bioinformatics \cite{BatQuoKul14} and neuroscience \cite{SnoZemAda13}. What makes such models appealing is that they exhibit repulsive behavior and lend themselves naturally to tasks where returning a diverse set of objects is important. 

One way to define a DPP is through an $N \times N$ symmetric positive semidefinite matrix $K$, called a \emph{kernel}, whose eigenvalues are bounded in the range $[0, 1]$.  Then the DPP associated with $K$, which we denote by $\DPP(K)$, is the distribution on $Y \subseteq [N]$ that satisfies, for any $J \subseteq [N]$,
$$ \mathbb{P}[J \subseteq Y] = \det(K_J),$$ 
where $K_J$ is the principal submatrix of $K$ indexed by the set $J$. The \emph{graph  induced by $K$} is the graph $G=([N], E)$ on the vertex set $[N]=\{1, \ldots, N\}$ that connects $i,j \in [N]$ if and only if $K_{i,j} \neq 0$.

There are fast algorithms for sampling (or approximately sampling) from $\DPP(K)$ \cite{deshpander, rebeschini2015fast, li2016fast, LiJegSra16}. 
Also marginalzing the distribution on a subset $I \subseteq [N]$ and conditioning on the event that $J \subseteq Y$ both result in new DPPs and closed form expressions for their kernels are known \cite{borodin2005eynard}.

There has been much less work on the problem of learning the parameters of a DPP. A variety of heuristics have been proposed, including Expectation-Maximization \cite{GilKulFox14}, MCMC \cite{AffFoxAda14}, and fixed point algorithms \cite{MarSra15}. All of these attempt to solve a nonconvex optimization problem and no guarantees on their statistical performance are known. Recently, Brunel \etal \cite{brunel2017maximum} studied the rate of estimation achieved by the maximum likelihood estimator, but the question of efficient computation remains open.

Apart from positive results on sampling, marginalization and conditioning, most provable results about DPPs are actually negative. It is conjectured that the maximum likelihood estimator is \textsf{NP}-hard to compute \cite{Kul12}. Actually, approximating the mode of size $k$ of a DPP to within a $c^k$ factor  is known to be \textsf{NP}-hard for some $c>1$ \cite{ccivril2009selecting, summa2015largest}. The best known algorithms currently obtain a $e^k + o(k)$ approximation factor \cite{nikolov2015randomized, nikolov2016maximizing}.

In this work, we bypass the difficulties associated with maximum likelihood estimation by using the \emph{method of moments} to achieve optimal sample complexity. We exhibit a parameter $\ell$ that we call the {\em cycle sparsity} of the graph induced by the kernel $K$ which governs the number of moments that need to be considered and, thus, the sample complexity. Moreover, we use a refined version of Horton's algorithm \cite{horton1987polynomial,amaldi2010efficient} to implement the method of moments in polynomial time.

The cycle sparsity of a graph is the smallest integer $\ell$ so that the cycles of length at most $\ell$ yield a basis for the cycle space of the graph. Even though there are in general exponentially many cycles in a graph to consider, Horton's algorithm constructs a minimum weight cycle basis and, in doing so, also reveals the parameter $\ell$ together with a collection of at most $\ell$ induced cycles spanning the cycle space. 

We use such cycles in order to construct our method of moments estimator. For any fixed $\ell \ge 2$, our overall algorithm has sample complexity $$ n = O\Big( \big(\frac{C}{\alpha}\big)^{2\ell}+ \frac{\log N}{\alpha^2 \varepsilon^2} \Big)$$
for some constant $C > 1$ and runs in time polynomial in $n$ and $N$, and learns the parameters up to an additive $\varepsilon$ with high probability. The $(C/\alpha)^{2\ell}$ term corresponds to the number of samples needed to recover the signs of the entries in $K$. We complement this result with a minimax lower bound (Theorem~\ref{ThmLB}) to show that this sample complexity is in fact near optimal. In particular, we show that there is an infinite family of graphs with cycle sparsity $\ell$ (namely length $\ell$ cycles) on which any algorithm requires at least $(C'\alpha)^{-2\ell}, C'>1$ samples  to recover the signs of the entries of $K$. Finally, we show experimental results that confirm many quantitative aspects of our theoretical predictions. Together, our upper bounds, lower bounds, and experiments present a nuanced understanding of which DPPs can be learned provably and efficiently.

\section{Estimation of the Kernel}
\label{SEC:EstKer}

\subsection{Model and definitions}
Let $Y_1, \ldots, Y_n$ be $n$ independent copies of $Y\sim \DPP(K)$, for some unknown kernel $K$ such that $0 \preceq K \preceq I_N$.  It is well known that $K$ is identified by $\DPP(K)$ only up to flips of the signs of its rows and columns: If $K'$ is another symmetric matrix with $0\preceq K' \preceq I_N$, then $\DPP(K')$=$\DPP(K)$ if and only if $K'=DKD$ for some $D\in\mathcal D_N$, where $\cD_N$ denotes the class of all $N\times N$ diagonal matrices with only $1$ and $-1$ on their diagonal \cite[Theorem 4.1]{Kul12}. We call such a transform a $\mathcal D_N$-similarity of $K$.

In view of this equivalence class, we define the following pseudo-distance between kernels $K$ and $K'$:
$$
\rho(K,K')=\inf_{D\in\mathcal D_N}|DKD-K'|_{\infty}\,,
$$
where for any matrix $K$, $|K|_{\infty}=\max_{i,j \in [N]}|K_{i,j}|$ denotes the entrywise sup-norm.



For any $S \subset [N]$, write $\Delta_S=\det(K_S)$ where $K_S$ denotes the $|S|\times |S|$ submatrix of $K$ obtained by keeping rows and colums with indices in $S$.  Note that for $1\leq i\neq j \leq N$, we have the following relations:
$$K_{i,i}=\rP[i\in Y], \quad \quad \Delta_{\{i,j\}}=\rP[\{i,j\}\subseteq Y],$$
and $$|K_{i,j}|=\sqrt{K_{i,i}K_{j,j}-\Delta_{\{i,j\}}}.$$
Therefore, the principal minors of size one and two of $K$  determine $K$ up to the sign of its off diagonal entries. 
In fact, for any $K$, there exists an $\ell$, depending only on the graph $G_K$ induced by $K$, such that $K$ can be recovered up to a $\mathcal D_N$-similarity with only the knowledge of its principal minors of size at most $\ell$. We will show that this $\ell$ is exactly the cycle sparsity.

\subsection{DPPs and graphs}
In this section, we review some of the interplay between graphs and DPPs that play a key role in the definition of our estimator.


We begin by recalling some standard graph theoretic notions. Let $G=([N],E)$, $|E|=m$. A cycle $C$ of $G$ is any connected subgraph in which each vertex has even degree. Each cycle $C$ is associated with an incidence vector $x \in GF(2)^m$ such that $x_e=1$ if $e$ is an edge in $C$ and $x_e=0$ otherwise.  The \emph{cycle space} $\cC$ of $G$ is the subspace of $GF(2)^m$ spanned the incidence vectors of the cycles in $G$. The dimension $\nu_G$ of the cycle space is called {\it cyclomatic number} and it is well known that $\nu_G :=m-N+\kappa(G)$, where $\kappa(G)$ denotes the number of connected components of $G$.

Recall that a {\it simple cycle} is a graph where every vertex has either degree two or zero and the set of vertices with degree two form a connected set. A \emph{cycle basis} is a basis of $\mathcal{C} \subset GF(2)^m$ such that every element is a simple cycle. It is well known that every cycle space has a cycle basis of induced cycles.


%
%
%
%

\begin{definition}
The \emph{cycle sparsity} of a graph $G$ is the minimal $\ell$ for which $G$ admits a cycle basis of induced cycles of length at most $\ell$, with the convention that $\ell = 2$ whenever the cycle space is empty. A corresponding cycle basis is called a \emph{shortest maximal cycle basis}.
\end{definition}

A \emph{shortest maximal cycle basis} of the cycle space was also studied for other reasons by \cite{CHICKERING199555}. We defer a discussion of computing such a basis to Section~\ref{SEC:Alg}.
   


%


For any subset $S \subseteq [N]$, denote by $G_K(S) = (S, E(S))$ the subgraph of $G_K$ induced by $S$. A matching of $G_K(S)$ is a subset $M\subseteq E(S)$ such that any two distinct edges in $M$ are not adjacent in $G(S)$. The set of vertices incident to some edge in $M$ is denoted by $V(M)$. We denote by $\mathcal M(S)$ the collection of all matchings of $G_K(S)$. Then, if $G_K(S)$ is an induced cycle, we can write the principal minor $\Delta_S=\det(K_S)$ as follows:
\begin{align}
\label{matchings}
\Delta_S &= \sum_{M\in\mathcal{M}(S)} (-1)^{|M| }  \prod_{\{i,j\} \in M } K_{i,j}^2 \prod_{i \not \in V(M)} K_{i,i} + 2 \times (-1)^{ |S| +1 } \prod_{\{i,j\} \in E(S)} K_{i,j}.
\end{align}

Others have considered the relationship between the principal minors of $K$ and recovery of $\DPP(K)$. There has been work regarding the \emph{symmetric principal minor assignment problem}, namely the problem of computing a matrix given an oracle that gives any principal minor in constant time \cite {Rising2015126}.

In our setting, we can approximate the principal minors of $K$ by empirical averages. However the accuracy of our estimator deteriorates with the size of the principal minor and we must therefore estimate the smallest possible principal minors in order to achieve optimal sample complexity. Here, we prove a new result, namely, that the smallest $\ell$ such that all the principal minors of $K$ are uniquely determined by those of size at most $\ell$ is exactly the cycle sparsity of the graph induced by $K$.

\begin{proposition}
Let $K \in \R^{N \times N}$ be a symmetric matrix, $G_K$ be the graph induced by $K$, and $\ell \ge 3$ be some integer. The kernel $K$ is completely determined up to $\cD_N$ similarity by its principal minors of size at most $\ell$ if and only if the cycle sparsity of  $G_K$ is at most $\ell$. \end{proposition}

\begin{proof}
Note first that all the principal minors of $K$ completely determine $K$ up to a $\mathcal D_N$-similarity \cite[Theorem 3.14]{Rising2015126}. Moreover,  recall that principal minors of degree at most 2 determine the diagonal entries of $K$ as well as the magnitude of its off diagonal entries. In particular, given these principal minors, one only needs to recover the signs of the off diagonal entries of $K$. Let the \emph{sign} of cycle $C$ in $K$ be the product of the signs of the entries of $K$ corresponding to the edges of $C$.

Suppose $G_K$ has cycle sparsity $\ell$ and let $(C_1,\ldots,C_{\nu})$ be a cycle basis of $G_K$ where each $C_i, i \in [\nu]$ is an induced cycle of length at most $\ell$. We have already seen that the diagonal entries of $K$ and the magnitudes of its off-diagonal entries are determined by the principal minors of size one and two.  
By \eqref{matchings},  the sign of any $C_i, i \in [\nu]$  is completely determined by the principal minor $\Delta_S$, where $S$ is the set of vertices of $C_i$ and is such that $|S|\leq\ell$. Moreover, for $i \in [\nu]$, let  $x_i \in GF(2)^m$ denote the incidence  vector of $C_i$. By definition, the incidence vector $x$ of any cycle $C$ is given by $\sum_{i \in \cI} x_i$ for some subset $\cI \subset [\nu]$. Then, it is not hard to see that the sign of $C$ is then given by the product of the signs of $C_i, i \in \cI$ and thus by corresponding principal minors. In particular, the signs of all cycles are determined by the principal minors $\Delta_S$ with $|S|\leq \ell$. In turn, by Theorem 3.12 in \cite{Rising2015126}, the signs of all cycles completely determine $K$, up to a $\mathcal D_N$-similarity.
%
%
%

Next, suppose the cycle sparsity of $G_K$ is at least $\ell+1$ and let $\cC_\ell$ be the subspace of $GF(2)^m$ spanned by the induced cycles of length at most $\ell$ in $G_K$. Let $x_1,\ldots,x_\nu$ be a basis of $\cC_\ell$ made of the incidence column vectors of induced cycles of length at most $\ell$ in $G_K$ and form the matrix $A \in GF(2)^{m \times \nu}$ by concatenating the $x_i$'s. Since $\cC_\ell$ does not span the cycle space of $G_K$, $\nu<\nu_{G_K}\leq m$. Hence, the rank of $A$ is less than $m$, so the null space of $A^\top$ is non trivial. Let $\bar x$ be the incidence column vector of an induced cycle $\bar C$ that is not in $\cC_\ell$ and let $h\in GL(2)^m$ with $A^\top h=0$, $h\neq 0$ and $\bar x^\top h=1$. These three conditions are compatible because $\bar C\notin\cC_\ell$. 
We are now in a position to define an alternate kernel $K'$ as follows: let $K'_{i,i}=K_{i,i}$ and $|K'_{i,j}|=|K_{i,j}|$ for all $i,j \in [N]$. We define the signs of the off diagonal entries of $K'$ as follows: for all edges $e=\{i,j\}, i \neq j$,  $\sgn(K'_{e})=\sgn(K_{e})$ if $h_e=0$ and $\sgn(K'_{e})=-\sgn(K_{e})$ otherwise. We now check that $K$ and $K'$ have the same principal minors of size at most $\ell$ but differ on a principal minor of size larger than $\ell$. To that end, let $x$ be the incidence vector of a cycle $C$ in $\cC_\ell$ so that $x=Aw$ for some $w \in GL(2)^{\nu}$. Thus the sign of $C$ in $K$ is given by
$$
\prod_{e\,:\,x_e=1}K_e=(-1)^{x^\top h}\prod_{e\,:\,x_e=1}K'_e=(-1)^{w^\top A^\top h}\prod_{e\,:\,x_e=1}K'_e=\prod_{e\,:\,x_e=1}K'_e
$$
because $A^\top h=0$. Therefore, the sign of any $C \in \cC_{\ell}$ is the same in $K$ and $K'$. Now, let $S\subseteq [N]$ with $|S|\leq \ell$ and let $G=G_{K_S}=G_{K'_S}$ be the graph corresponding to $K_S$ (or, equivalently, to $K_S'$). For any induced cycle $C$ in $G$, $C$ is also an induced cycle in $G_K$ and its length is at most $\ell$. Hence, $C\in\cC_\ell$ and the sign of $C$ is the same in $K$ and $K'$. By \cite[Theorem 3.12]{Rising2015126}, $\det(K_S)=\det(K'_S)$.
Next observe that the sign of $\bar C$ in $K$ is given by
$$
\prod_{e\,:\,\bar x_e=1}K_e=(-1)^{\bar x^\top h}\prod_{e\,:\,\bar x_e=1}K'_e=-\prod_{e\,:\,x_e=1}K'_e
$$
Note also that since $\bar C$ is an induced cycle of $G_K=G_{K'}$, the above quantity is nonzero. Let $\bar S$ be the set of vertices in $\bar C$. By~\eqref{matchings} and the above display, we have $\det(K_{\bar S}) \neq \det(K'_{\bar S})$. Together with~\cite[Theorem 3.14]{Rising2015126}, it yields $K \neq DK'D$ for all $D \in \cD_N$.

 \end{proof}

%
%

%
%
%
%

\subsection{Definition of the Estimator}
\label{SUBSEC:DefEst}

Our procedure is based on the previous lemmata and can be summarized as follows. We first estimate the diagonal entries (i.e., the principal minors of size one) of $K$ by the method of moments. By the same method, we estimate the principal minors of size two of $K$, and we deduce estimates of the magnitude of the off-diagonal entries. To use these estimates to deduce an estimate $\hat G$ of $G_K$, we make the following assumption on the kernel $K$.
\begin{assumption} \label{Assumption1}
	Fix $\alpha \in (0,1)$. For all $1\leq i<j\leq N$, either $K_{i,j}=0$, or $|K_{i,j}|\geq \alpha$.
\end{assumption}
Finally, we find a shortest maximal cycle basis of $\hat G$ and we set the signs of our non-zero off-diagonal entry estimates by using estimators of the principal minors induced by the elements of the basis, again obtained by the method of moments.

For $S \subseteq [N]$, set 
$$
{\hat \Delta_S =\frac{1}{n}\sum_{p=1}^n \mathds 1_{S \subseteq Y_p}}\,,
$$
and define 
$$ \hat K_{i,i} = \hat \Delta_{\{i\}} \quad \text{ and } \quad \hat B_{i,j}=\hat K_{i,i}\hat K_{j,j}-\hat \Delta_{\{i,j\}},$$
where $\hat K_{i,i}$ and $\hat B_{i,j}$ are our estimators of $K_{i,i}$ and $K_{i,j}^2$, respectively. 

Define $\hat G = ([N], \hat E)$, where, for $i \ne j$, $\{i,j\} \in \hat E$ if and only if $\hat B_{i,j} \ge \frac{1}{2} \alpha^2$. The graph $\hat G$ is our estimator of $G_K$. Let $\{\hat C_1,...,\hat C_{\nu_{\hat G}}\}$ be a shortest maximal cycle basis of the cycle space of $\hat G$. Let $\hat S_i \subseteq [N]$ be the subset of vertices of $\hat C_i$, for $1 \le i \le \nu_{\hat G}$. We define
\begin{align*}
\hat H_i & = \hat \Delta_{\hat S_i}  - \sum_{M\in\mathcal{M}(\hat S_i)}(-1)^{|M|}  \prod_{\{i,j\} \in M } \hat B_{i,j}  \prod_{i \not \in V(M)} \hat K_{i,i} ,
\end{align*}
for $1\leq i\leq \nu_{\hat G}$. In light of \eqref{matchings}, for large enough $n$, this quantity should be close to 
$$
H_i=2 \times (-1)^{ |\hat S_i| +1 } \prod_{\{i,j\} \in E(\hat S_i)} K_{i,j}\,. 
$$
We note that this definition is only symbolic in nature, and computing $\hat H_i$ in this fashion is extremely inefficient. Instead, to compute it in practice, we will use the determinant of an auxiliary matrix, computed via a matrix factorization. Namely, let us define the matrix $\widetilde K \in \R^{N \times N}$ such that $\widetilde K_{i,i} = \hat K_{i,i}$ for $1\le i \le N$, and $\widetilde K_{i,j} = \hat B^{1/2}_{i,j}$. We have
\begin{align*}
\det \widetilde K_{\hat S_i} &= \sum_{M\in \mathcal{M}} (-1)^{|M|}\prod_{\{i,j\} \in M } \hat B_{i,j}   \prod_{i \not \in V(M)} \hat K_{i,i}  + 2 \times (-1)^{ |\hat S_i| +1 }  \prod_{\{i,j\} \in \hat E(\hat S_i)} \hat B^{1/2}_{i,j},
\end{align*}
so that we may equivalently write
$$\hat H_i = \hat \Delta_{\hat S_i} - \det ( \widetilde K_{\hat S_i} ) + 2 \times (-1)^{|\hat S_i|+1} \prod_{\{i,j\} \in \hat E(\hat S_i)} \hat B_{i,j}^{1/2}.$$



Finally, let $\hat m=|\hat E|$. Set the matrix $A\in GF(2)^{\nu_{\hat G} \times \hat m}$ with $i$-th row representing $\hat C_i$ in $GF(2)^m$, $1\leq i\leq \nu_{\hat G}$, $b=(b_1,\ldots,b_{\nu_{\hat G}}) \in GF(2)^{\nu_{\hat G}}$ with $b_i = \frac{1}{2}[ \sgn(\hat H_i) +1]$, $1\le i \le \nu_{\hat G}$, and let $x \in GF(2)^m$ be a solution to the linear system $A x = b$ if a solution exists, $x = \mathds 1_m$ otherwise.

We define $\hat K_{i,j}=0$ if $\{i,j\}\notin \hat E$ and $\hat K_{i,j} =  \hat K_{j,i} = (2 x_{\{i,j\}} -1 ) \hat B^{1/2}_{i,j}$ for all $\{i,j\} \in \hat E$. 

\subsection{Geometry}
%
%

The main result of this subsection is the following lemma which relates the quality of estimation of $K$ in terms of  $\rho$ to  the quality of estimation 
of the principal minors $\Delta_S=\det(K_S)$, $S \subset [N]$.

\begin{lemma}
\label{thm:paramdist}
Let $K$ satisfy Assumption \ref{Assumption1} and let $\ell$ be the cycle sparsity of $G_K$. Let $\varepsilon>0$. If $| \hat \Delta_{S} - \Delta_{S} | \leq \varepsilon$ for all $S\subseteq [N]$ with $|S|\le 2$ and if $|\hat \Delta_{S} - \Delta_{S} | \leq (\alpha/4)^{|S|}$ for all $S\subseteq [N]$ with $3\leq |S|\leq \ell$, then
$$
{\rho(\hat K, K) < \frac{4 \varepsilon}{\alpha}}.
$$
\end{lemma}

\begin{proof}

We can bound $| \hat B_{i,j} - K_{i,j}^2|$, namely,
\begin{align*}
	\hat B_{i,j} & \leq (K_{i,i} + \alpha^2/16)(K_{j,j} + \alpha^2/16) - (\Delta_{\{i,j\}} - \alpha^2/16)  \leq K^2_{i,j} + \alpha^2/4
\end{align*}
and
\begin{align*}
	\hat B_{i,j} & \geq (K_{i,i} - \alpha^2/16)(K_{j,j} - \alpha^2/16) - (\Delta_{\{i,j\}} + \alpha^2/16) \geq K^2_{i,j} - 3 \alpha^2/16,
\end{align*}
giving $\displaystyle{| \hat B_{i,j} - K_{i,j}^2| < \alpha^2/4 }$. Thus, we can correctly determine whether $K_{i,j} = 0$ or $|K_{i,j}| \ge \alpha$, yielding $\hat G=G_K$. In particular, the cycle basis $\hat C_1,\ldots,\hat C_{\nu_{\hat G}}$ of $\hat G$ is a cycle basis of $G_K$. Let $1\leq i\leq \nu_{\hat G}$. Denote by $t=(\alpha/4)^ {|S_i|}$. We have
\begin{align*}
 \Big | \hat H_i - H_i \Big| & \le |  \mathcal{M}(\hat S_i)| \max_{x \in \pm 1} \left[ (1 + 4 t x)^{|\hat S_i|} -1 \right] \\ 
 & \le |  \mathcal{M}(\hat S_i)| \left[ (1 + 4 t)^{|\hat S_i|} -1 \right] \\ 
 & \le T \left(|\hat S_i|, \left\lfloor \frac{|\hat S_i|}{2} \right\rfloor \right) 4 t \; T( |\hat S_i| , |\hat S_i| ) \\ 
 & \le 4t \; ( 2^{\frac{|\hat S_i|}{2}} - 1) ( 2^{|\hat S_i|} - 1 )  \le t 2^{2|\hat S_i|}=\alpha^{|\hat S_i|},
\end{align*}
where, for positive integers $p<q$, we denote by $T(q,p) = \sum_{i=1}^p {q \choose i}$. Therefore, we can determine the sign of the product $$\displaystyle{\prod_{\{i,j\} \in E(\hat S_i)} K_{i,j}}$$ for every element in the cycle basis and recover the signs of the non-zero off-diagonal entries of $K_{i,j}$. Hence, $\displaystyle{\rho(\hat K,K)=\max_{1\leq i,j\leq N}\left||\hat K_{i,j}|-|K_{i,j}|\right|}$. For $i=j$, $\displaystyle{\left||\hat K_{i,j}|-|K_{i,j}|\right|=|\hat K_{i,i}-K_{i,i}|\leq \varepsilon}$. For $i\neq j$ with $\{i,j\}\in \hat E=E$, one can easily show that $\displaystyle{\left|\hat B_{i,j}-K^2_{i,j}\right|\leq 4\varepsilon}$, yielding
\begin{align*}
| \hat B^{1/2}_{i,j} - | K_{i,j} | | &\le \frac{ 4\varepsilon }{ | \hat B^{1/2}_{i,j} + | K_{i,j} | | } \le \frac{4\varepsilon}{\alpha},
\end{align*}
which completes the proof.
\end{proof}

We are now in a position to establish a sufficient sample size to estimate $K$ within distance $\varepsilon$.

\begin{theorem} \label{ThmUB}
Let $K$ satisfy Assumption \ref{Assumption1} and let $\ell$ be the cycle sparsity of $G_K$. Let $\varepsilon>0$. For any $A>0$, there exists $C>0$ such that 
$$
 n \ge C\Big(\frac{1}{\alpha^2\varepsilon^{-2}}+\ell\big(\frac{4}{\alpha}\big)^{2\ell}\Big)\log N\,,
 $$
yields $\rho(\hat K, K)\le \varepsilon$ with probability at least $1-N^{-A}$.
\end{theorem}

\begin{proof}
Using the previous lemma, and applying a union bound,
\begin{align}
	\mathbb{P}\left[\rho(\hat K,K)>\varepsilon\right] & \leq \sum_{|S|\leq 2}\mathbb{P}\left[|\hat \Delta_S-\Delta_S|>\alpha\varepsilon/4\right] +\sum_{2\leq |S|\leq \ell}\mathbb{P}\left[|\hat \Delta_S-\Delta_S|>(\alpha/4)^{|S|}\right] \nonumber \\
	\label{DevIneqThm1} & \leq 2N^2e^{-n\alpha^2\varepsilon^2/8}+2N^{\ell+1}e^{-2n(\alpha/4)^{2\ell}},
\end{align}
where we used Hoeffding's inequality. The desired result follows.
\end{proof}

\vspace{3mm}

\section{Information theoretic lower bound}

Next, we prove an information-theoretic lower that holds already if $G_K$ is a  cycle of length $\ell$. 
\begin{lemma} \label{LemmaKL}
For $\eta\in\{-,+\}$, let $K^\eta$ be the $\ell \times \ell$ matrix with elements given by
$$
K_{i,j}=\left\{ \begin{array}{ll}
1/2 & \text{if} \ j=i\\
\alpha & \text{if} \ j=i\pm 1\\
\eta \alpha & \text{if} \ (i,j) \in \{(1,\ell), (\ell,1)\}\\
0 & \text{otherwise}
\end{array}\right.
$$
%
%
%
%
%
%
%

Let $D(K \|K')$ and $\mathbb H(K,K')$ denote respectively the Kullback-Leibler divergence  and the Hellinger distance between $\DPP(K)$ and $\DPP(K')$. Then, for any $\alpha\leq 1/8$, it holds
$$
D(K \|K') \le 4(6\alpha)^\ell,\qquad \text{and} \qquad \mathbb H(K,K') \le (8\alpha^2)^\ell\,.
$$
\end{lemma}

\begin{proof}
It is straightforward to see that
$$
\det(K_J^+)-\det(K_J^-) =
\begin{cases}
2 \alpha^\ell \quad \mbox{if $J = [\ell]$} \\
\mbox{ }0 \quad \mbox{ } \mbox{ }\mbox{ else}
\end{cases}
$$
If $Y$ is sampled from $\DPP(K^\eta)$, we denote by $p_\eta(S)=\mathbb P[Y=S]$, for $S\subseteq [\ell]$. It follows from the inclusion-exclusion principle that for all $S\subseteq [\ell]$,
\begin{align}
p_+(S) - p_-(S) & = \sum_{J\subseteq [\ell]\setminus S} (-1)^{|J|}(\det K_{S\cup J}^+-\det K_{S\cup J}^-)
 \nonumber \\
&  = (-1)^{\ell-|S|}(\det K^+ -\det K^-) \nonumber\\
&=\pm 2\alpha^\ell\,,\label{KL1}
\end{align}
where $|J|$ denotes the cardinality of $J$. 
The inclusion-exclusion principle also yields that $p_\eta(S)=|\det(K^\eta-I_{\bar S})|$ for all $S\subseteq [l]$, where $I_{\bar S}$ stands for the $\ell\times\ell$ diagonal matrix with ones on its entries $(i,i)$ for $i\notin S$, zeros elsewhere.

Denote by $D(K^+\|K^-)$ the Kullback Leibler divergence between $\DPP(K^+)$ and $\DPP(K^-)$:
\begin{align}
        D(K^+\|K^-) & = \sum_{S\subseteq [\ell]} p_+(S)\log\left(\frac{p_+(S)}{p_-(S)}\right) \nonumber \\
        & \leq \sum_{S\subseteq [\ell]}\frac{p_+(S)}{p_-(S)}(p_+(S)-p_-(S)) \nonumber \\
        \label{KL2} & \leq 2\alpha^\ell \sum_{S\subseteq [\ell]}\frac{|\det(K^+ -I_{\bar S})|}{|\det(K^- -I_{\bar S})|},
\end{align}
by \eqref{KL1}. Using the fact that $0<\alpha\leq 1/8$ and the Gershgorin circle theorem we conclude that the absolute value of all eigenvalues of $K^\eta-I_{\bar S}$ are between $1/4$ and $3/4$. Thus we obtain from \eqref{KL2} the bound $\displaystyle{D(K^+\|K^-) \leq 4(6\alpha)^\ell}$.

Using the same arguments as above, the Hellinger distance $\mathbb H(K^+,K^-)$ between $\DPP(K^+)$ and $\DPP(K^-)$ satisfies
\begin{align*}
	\mathbb H(K^+,K^-) = \sum_{J\subseteq [\ell]}\left(\frac{p_+(J)-p_-(J)}{\sqrt{p_+(J)}+\sqrt{p_-(J)}}\right)^2 \leq \sum_{J\subseteq [\ell]} \frac{4\alpha^{2\ell}}{2\cdot 4^{-\ell}} = (8\alpha^2)^\ell.
\end{align*}
\end{proof}


The sample complexity lower bound now follows from standard arguments.

\begin{theorem} \label{ThmLB}
Let $0<\varepsilon\leq \alpha\leq 1/8$ and $3 \leq \ell \leq N$. There exists a constant $C>0$ such that if 
$$
n \le C\Big ( \frac{8^\ell}{\alpha^{2\ell}}+\frac{\log (N/\ell)}{(6\alpha)^\ell} + \frac{\log N}{\varepsilon^2} \Big )
$$
then the following holds: for any estimator $\hat K$ based on $n$ samples, there exists a kernel $K$ that satisfies Assumption~\ref{Assumption1} and such that the cycle sparsity of $G_K$ is $\ell$ and for which $\rho(\hat K, K) \ge \varepsilon$ with probability at least $1/3$.
\end{theorem}

\begin{proof}
Recall the notation of Lemma \ref{LemmaKL}. For the first term, consider the $N\times N$ block diagonal matrix $K$ (resp. $K'$) where its first block is $K^+$ (resp. $K^-$) and its second block is $I_{N-\ell}$. By a standard argument, the Hellinger distance $\mathbb H_n(K,K')$ between the product measures $\DPP(K)^{\otimes n}$ and $\DPP(K')^{\otimes n}$ satisfies 
\begin{align*}
	1-\frac{\mathbb H_n^2(K,K')}{2} & = \big(1-\frac{\mathbb H^2(K,K')}{2}\big)^n \geq \big(1-\frac{\alpha^{2\ell}}{2 \times 8^\ell}\big)^n,
\end{align*}
yielding the first term. 

By padding with zeros, we can assume that $L = N/\ell$ is an integer. Let $K^{(0)}$ be a block diagonal matrix where each block is $K^+$ (using the notation of Lemma~\ref{LemmaKL}). For $j=1,\ldots,L$, define the $N\times N$ block diagonal matrix $K^{(j)}$ as the matrix obtained from $K^{(0)}$ by replacing its $j$th block with $K^-$ (again using the notation of Lemma~\ref{LemmaKL}).

Since $\DPP(K^{(j)})$ is the product measure of $L$ lower dimensional DPPs which are each independent of each other, using
Lemma \ref{LemmaKL} we have $$\displaystyle{D(K^{(j)}\|K^{(0)}) \leq 4(6\alpha)^\ell}.$$ Hence, by Fano's lemma (see, e.g., Corollary 2.6 in \cite{Tsybakov2009}), the sample complexity to learn the kernel of a DPP within a distance $\varepsilon\leq \alpha$ is $$\displaystyle{\Omega\left(\frac{\log(N/\ell)}{(6\alpha)^\ell}\right)}.$$

The third term follows from considering $K_0=(1/2)I_N$ and letting $K_j$ be obtained from $K_0$ by adding $\varepsilon$ to the $j$th entry along the diagonal. It is easy to see that $\displaystyle{D(K_j\|K_0) \leq 8\varepsilon^2}$. Hence, a second application of Fano's lemma yields that the sample complexity to learn the kernel of a DPP within a distance $\varepsilon$ is $\Omega( \frac{\log N}{\varepsilon^2})$ which completes the proof.
\end{proof}

The third term in the lower bound is the standard parametric term and it is unavoidable in order to estimate the magnitude of the coefficients of $K$. The other terms are more interesting. They reveal that the cycle sparsity of $G_K$, namely, $\ell$, plays a key role in the task of recovering the sign pattern of $K$. Moreover the theorem shows that the sample complexity of our method of moments estimator is near optimal.

\vspace{3mm}

\section{Algorithm}
\label{SEC:Alg}

We now detail an algorithm to compute the estimator $\hat K$ defined in Section \ref{SEC:EstKer}. It is well known that a cycle basis of minimum total length can be computed in polynomial time. Horton \cite{horton1987polynomial} gave an algorithm, now referred to as Horton's algorithm, to compute such a cycle basis in $O(m^3N)$ time by carefully choosing a polynomial number of cycles and performing Gaussian elimination on them. There have been several improvements of Horton's algorithm, notably one that runs in $O( m^2 N [\log N]^{-1})$ time \cite{amaldi2010efficient}. In addition,  it is known that a cycle basis of minimum total length is a shortest maximal cycle basis~\cite{CHICKERING199555}. These results implying the following.

\begin{lemma}
\label{cor:horton}
Let $G = ([N],E)$, $|E|=m$. There exists an algorithm to compute a shortest maximal cycle basis in $O( m^2 N [\log N]^{-1})$ time.
\end{lemma}

In addition, we recall the following standard result regarding the complexity of Gaussian elimination \cite{golub2012matrix}.

\begin{lemma}
\label{lem:gauss}
Let $A \in GF(2)^{\nu \times m}$, $b \in GF(2)^{\nu}$. Then Gaussian elimination will find a vector $x \in GF(2)^m$ such that $Ax =b$ or conclude that none exists in $O(\nu^2 m)$ time.
\end{lemma}

We give our procedure for computing the estimator $\hat K$ in Algorithm \ref{alg:mainalg}. In the following theorem, we give the complexity of Algorithm \ref{alg:mainalg} and establish an upper bound on both the required sample complexity for the recovery problem and the distance between $K$ and $\hat K$.

\begin{algorithm}[tb]
   \caption{Compute Estimator $\hat K$}
   \label{alg:mainalg}
\begin{algorithmic}
   \STATE {\bfseries Input:} samples $Y_1,...,Y_n$, parameter $\alpha>0$.
      \vspace{0.5pc}
     \STATE Compute $\hat \Delta_{S}$ for all $|S| \le 2$.
     \STATE Set $\hat K_{i,i} = \hat \Delta_{\{i\}}$ for $1 \le i \le N$.
     \STATE Compute $\hat B_{i,j}$ for $1 \le i < j \le N$.
     \STATE Form $\widetilde K \in \R^{N \times N}$ and $\hat G = ([N],\hat E)$.
     \STATE Compute a shortest maximal cycle basis $\{\hat v_1,...,\hat v_{\nu_{\hat G}} \}$.
     \STATE Compute $\hat \Delta_{\hat S_i}$ for $1 \le i \le  \nu_{\hat G} $.
     \STATE Compute $\hat C_{\hat S_i}$ using $\det \widetilde K_{ \hat S_i }$ for $1 \le i \le  \nu_{\hat G} $.
     \STATE Construct $A \in GF(2)^{ \nu_{\hat G}  \times m}$, $b \in GF(2)^{ \nu_{\hat G} }$.
     \STATE Solve $Ax = b$ using Gaussian elimination.
     \STATE Set $\hat K_{i,j} = \hat K_{j,i} = (2 x_{\{i,j\}} -1 ) \hat B_{i,j}^{1/2}$, for all $\{i,j\} \in \hat E$.\end{algorithmic}
\end{algorithm}

\begin{theorem}
Let $K \in \R^{N \times N}$ be a symmetric matrix satisfying $0 \preceq K \preceq I$, and satisfying Assumption \ref{Assumption1}. Let $G_K$ be the graph induced by $K$ and $\ell$ be the cycle sparsity of $G_K$. Let $Y_1,...,Y_n$ be samples from DPP($K$) and $\delta\in (0,1)$. If 
$$n >  \frac{ \log ( N^{\ell+1} / \delta)}{\left(\alpha/4 \right)^{2 \ell}},$$
then with probability at least $1- \delta$, Algorithm \ref{alg:mainalg} computes an estimator $\hat K$ which recovers the signs of $K$ up to a $\mathcal D_N$-similarity and satisfies 
\begin{equation} \label{devineqThm3} \rho(K,\hat K)< \frac{1}{\alpha} \left( \frac{8\log (4N^{\ell+1}/\delta)}{n} \right)^{1/2}
\end{equation}
 in $O(m^3) + O(nN^2)$ time.
\end{theorem}

\begin{proof}
\eqref{devineqThm3} follows directly from \eqref{DevIneqThm1} in the proof of Theorem \ref{ThmUB}. That same proof also shows that with probability at least $1-\delta$, the support of $G_K$ and the signs of $K$ are recovered up to a $\mathcal D_N$-similarity. What remains is to upper bound the worst case run time of Algorithm \ref{alg:mainalg}. We will perform this analysis line by line. Initializing $\hat K$ requires $O(N^2)$ operations. Computing $\Delta_S$ for all subsets $|S| \le 2$ requires $O(nN^2)$ operations. Setting $\hat K_{i,i}$ requires $O(N)$ operations. Computing $\hat B_{i,j}$ for $1 \le i < j \le N$ requires $O(N^2)$ operations. Forming $\widetilde K$ requires $O(N^2)$ operations. Forming $G_K$ requires $O(N^2)$ operations. By Lemma \ref{cor:horton}, computing a shortest maximal cycle basis requires $O(mN)$ operations. Defining the subsets $S_i$, $1 \le i \le  \nu_{\hat G} $, requires $O(mN)$ operations. Computing $\hat \Delta_{S_i}$ for $1 \le i \le  \nu_{\hat G} $ requires $O(nm)$ operations. Computing $\hat C_{S_i}$ using $\det(\widetilde K [S_i] )$ for $1 \le i \le \nu_{\hat G} $ requires $O(m\ell^3)$ operations, where a factorization of each $\widetilde K [ S_i]$ is used to compute each determinant in $O(\ell^3)$ operations. Constructing $A$ and $b$ requires $O(m \ell)$ operations. By Lemma \ref{lem:gauss}, finding a solution $x$ using Gaussian elimination requires $O(m^3)$ operations. Setting $\hat K_{i,j}$ for all edges $\{i,j\} \in E$ requires $O(m)$ operations. Considered together, it implies that Algorithm \ref{alg:mainalg} requires $O(m^3) + O(nN^2)$ operations.
\end{proof}

\subsection{Chordal Graphs}

Now we show that when $G_K$ has a special structure, there exists an $O(m)$ algorithm to determine the signs of the off-diagonal entries of the estimator $\hat K$, resulting in an improved overall runtime of $O(m) + O(n N^2)$. Recall first that a graph $G = ([N],E)$ is said to be \emph{chordal} if every induced cycle in $G$ is of length three. Moreover,  a graph $G = ([N],E)$ has a \emph{perfect elimination ordering (PEO)} if there exists an ordering of the vertex set $\{v_1,...,v_N\}$ such that, for all $i$, the graph induced by $\{v_i\} \cup \{ v_j | \{i,j\} \in E, j > i\}$ is a clique.
It is well known that a graph $G = ([N],E)$ is chordal if and only if it has a PEO. A PEO of a chordal graph $G =([N],E)$, $|E| = m$, can be computing in $O(m)$ operations using lexicographic breadth-first search \cite{rose1976algorithmic}.

We prove the following result regarding PEOs of chordal graphs.

\begin{lemma}
Let $G = ([N],E)$, be a chordal graph and $\{v_1,...,v_n\}$ be a PEO. Given $i$, let $i^* := \min \{ j | j > i , \{ v_i ,v_j \} \in E \}$. Then the graph $G' = ([N], E')$, where $E' = \{ \{v_i,v_{i^*} \} \}_{i=1}^{N- \kappa(G)}$, is a spanning forest of $G$.
\end{lemma}

\begin{proof}
We first show that there are no cycles in $G'$. Suppose to the contrary, that there exists a induced cycle $C$ of length $k$ on the vertices $\{v_{j_1},...,v_{j_k} \}$. Let us choose the vertex of smallest index. This vertex is connected to two other vertices in the cycle of larger index. This is a contradiction to the construction. 

All that remains is to show that $|E'| = N-\kappa(G)$. It suffices to prove the case $\kappa(G) = 1$. Again, suppose to the contrary, that there exists a vertex $v_i$, $i <N$, with no neighbors of larger index. Let $P$ be the shortest path in $G$ from $v_i$ to $v_N$. By connectivity, such a path is guaranteed to exist. Let $v_k$ be the vertex of smallest index in the path. However, it has two neighbors in the path of larger index, which must be adjacent to each other. Therefore, there is a shorter path. This is a contradiction.
\end{proof}

Now, given the chordal graph $G_K$ induced by $K$ and the estimates of principal minors of size at most three, we provide an algorithm to determine the signs of the edges of $G_K$, or, equivalently, the off-diagonal entries of $K$.

\begin{algorithm}[tb]
   \caption{Compute Signs of Edges in Chordal Graph}
   \label{alg:chordal}
\begin{algorithmic}
   \STATE {\bfseries Input:}  $G_K = ([N],E)$ chordal, $\hat \Delta_{S}$ for $|S|\le 3$.
   \vspace{0.5pc}
      \STATE Compute a PEO $\{v_1,...,v_N\}$.
     \STATE Compute the spanning forest $G' = ([N], E')$.
     \STATE Set all edges in $E'$ to have positive sign.
     \STATE Compute $\hat C_{\{i,j,i^*\}}$ for all $\{i,j \} \in E \setminus E'$, $j<i$.
     \STATE Order edges $E \setminus E'  = \{e_{1} ,..., e_{\nu} \}$ such that $i > j$ if $\max e_{i} < \max e_{j}$.
     \STATE Visit edges in sorted order and for $e =\{i,j\}$, $j > i$, set $$\sgn(\{i,j\}) = \sgn(\hat C_{\{i,j,i^*\}}) \sgn(\{i,i^*\}) \sgn(\{j,i^*\})$$
     \end{algorithmic}
\end{algorithm}

\begin{theorem}
Algorithm \ref{alg:chordal} determines the signs of the edges of $G_K$ in $O(m)$ time.
\end{theorem}

\begin{proof}
We will simultaneously perform a count of the operations and a proof of the correctness of the algorithm. Computing a PEO requires $O(m)$ operations. Computing the spanning forest requires $O(m)$ operations. The edges of the spanning tree can be given arbitrary sign, because it is a cycle-free graph. Computing each $\hat C_{\{i,j,i^*\}}$ requires a constant number of operations because $\ell = 3$, requiring a total of $O(m-N)$ operations. Ordering the edges requires $O(m)$ operations. Setting the signs of each remaining edge requires $O(m)$ operations. For each three cycle used, the other two edges have already had their sign determined, by construction. 
\end{proof}

Therefore, for the case in which $G_K$ is chordal, the overall complexity required by our algorithm to compute $\hat K$ is reduced to $O(m) + O(n N^2)$.

\vspace{3mm}

\section{Experiments}

Here we present experiments to supplement the theoretical results of the paper. We test our algorithm for two random matrices. First, we consider the matrix $K \in \R^{N \times N}$ corresponding to the cycle on $N$ vertices, $$K = \frac{1}{2} I + \frac{1}{4} A,$$
where $A$ is symmetric, and has non-zero entries only on the edges of the cycle, either $+1$ or $-1$, each with probability $1/2$. By the Gershgorin circle theorem, $0 \preceq K \preceq I$. Next, we consider the matrix $K \in \R^{N \times N}$ corresponding to the clique on $N$ vertices,
$$K = \frac{1}{2} I + \frac{1}{4 \sqrt{N}} A,$$
where $A$ is symmetric, and has all entries either $+1$ or $-1$, each with probability $1/2$. It is well known that $- 2 \sqrt{N} \preceq  A \preceq 2 \sqrt{N}$ with high probability, implying $0\preceq K \preceq I$.

For both cases, and a range of values of matrix dimension $N$ and samples $n$, we randomly generate $50$ instances of $K$ and run our algorithm on each instance. We record the proportion of the time that we recover the  graph induced by $K$, and the proportion of the time we recover the  graph induced by $K$ and correctly determine the signs of the entries.

In Figure \ref{fig1}, the shade of each box represents the proportion of trials that were recovered successfully for a given pair $N,n$. A completely white box corresponds to zero success rate, black to a perfect success rate.

%
 
\begin{figure}
\begin{center}
\subfloat[graph recovery, cycle]{\includegraphics[width=.5\textwidth]{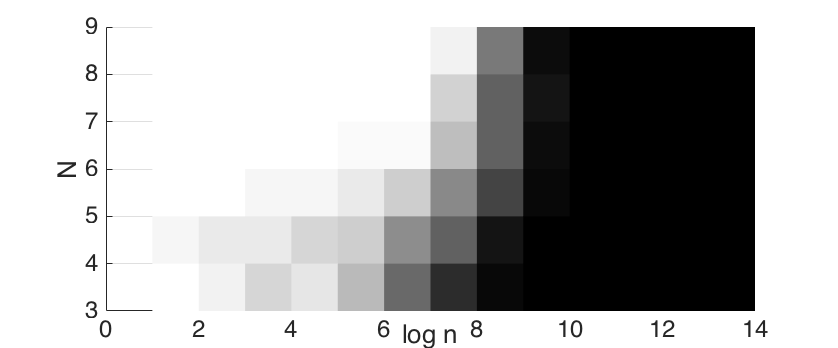}} 
\subfloat[graph and sign recovery, cycle]{\includegraphics[width=.5\textwidth]{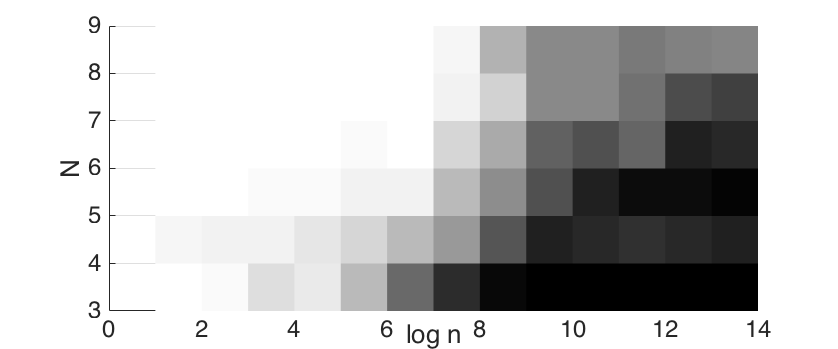}} \\
\subfloat[graph recovery, clique]{\includegraphics[width=.5\textwidth]{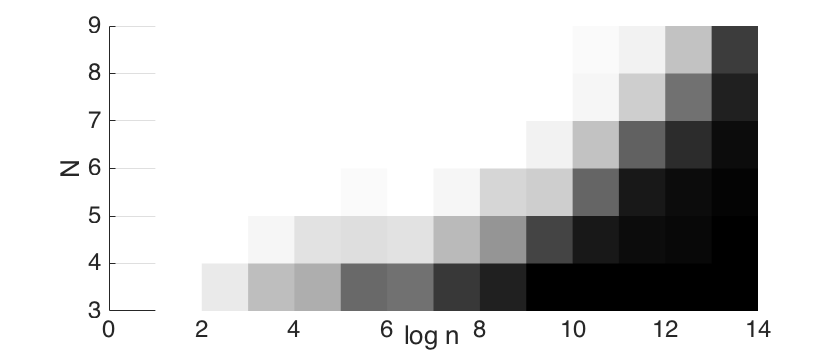}} 
\subfloat[graph and sign recovery, clique]{\includegraphics[width=.5\textwidth]{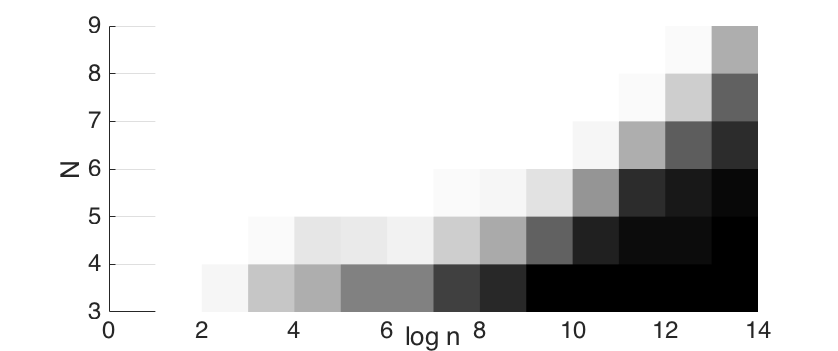}}
\caption{Plots of the proportion of successive graph recovery, and graph and sign recovery, for random matrices with cycle and clique graph structure, respectively. The darker the box, the higher the proportion of trials that were recovered successfully.}
\label{fig1}
\end{center}
\end{figure}   
%

The plots corresponding to the cycle and the clique are telling. We note that, conditional on successful recovery of the sparsity pattern of $K$, recovery of the signs of the off-diagonal entries of the full matrix quickly follows. However, for the cycle, there exists a noticeable gap between the number of samples required for recovery of the sparsity pattern and the number of samples required to recover the signs of the off-diagonal entries. This confirms through practice what we have already gleaned theoretically.

\section{Conclusion and open questions}

 In this paper, we gave the first provable guarantees for learning the parameters of a DPP. Our upper and lower bounds reveal the key role played by the parameter $\ell$, which is the cycle sparsity of graph induced by the kernel of the DPP. Our estimator does not need to know $\ell$ beforehand, but can adapt to the instance. Moreover, our procedure outputs an estimate of $\ell$ which could potentially be used for further inference questions such as testing and confidence intervals. An interesting open question is whether on a graph by graph basis, the parameter $\ell$ exactly determines the optimal sample complexity. Moreover when the number of samples is too small, can we exactly characterize which signs can be learned correctly and which cannot (up to a similarity transformation by $D$)? Such results would lend new theoretical insights into the output of algorithms for learning DPPs, and which individual parameters in the estimate we can be confident about and which we cannot be. 


\vspace{3mm}

\bibliography{DPPMOM}

\newcommand{\etalchar}[1]{$^{#1}$}
\begin{thebibliography}{BMRU17}

\bibitem[AFAT14]{AffFoxAda14}
Raja~Hafiz Affandi, Emily~B. Fox, Ryan~P. Adams, and Benjamin Taskar.
\newblock Learning the parameters of determinantal point process kernels.
\newblock In {\em Proceedings of the 31th International Conference on Machine
  Learning, {ICML} 2014, Beijing, China, 21-26 June 2014}, pages 1224--1232,
  2014.

\bibitem[AIR10]{amaldi2010efficient}
Edoardo Amaldi, Claudio Iuliano, and Romeo Rizzi.
\newblock Efficient deterministic algorithms for finding a minimum cycle basis
  in undirected graphs.
\newblock In {\em International Conference on Integer Programming and
  Combinatorial Optimization}, pages 397--410. Springer, 2010.

\bibitem[BMRU17]{brunel2017maximum}
Victor-Emmanuel Brunel, Ankur Moitra, Philippe Rigollet, and John Urschel.
\newblock Maximum likelihood estimation of determinantal point processes.
\newblock {\em arXiv:1701.06501}, 2017.

\bibitem[BQK{\etalchar{+}}14]{BatQuoKul14}
Nematollah~Kayhan Batmanghelich, Gerald Quon, Alex Kulesza, Manolis Kellis,
  Polina Golland, and Luke Bornn.
\newblock Diversifying sparsity using variational determinantal point
  processes.
\newblock {\em ArXiv: 1411.6307}, 2014.

\bibitem[BR05]{borodin2005eynard}
Alexei Borodin and Eric~M Rains.
\newblock Eynard--mehta theorem, schur process, and their pfaffian analogs.
\newblock {\em Journal of statistical physics}, 121(3):291--317, 2005.

\bibitem[CGH95]{CHICKERING199555}
David~M. Chickering, Dan Geiger, and David Heckerman.
\newblock On finding a cycle basis with a shortest maximal cycle.
\newblock {\em Information Processing Letters}, 54(1):55 -- 58, 1995.

\bibitem[{\c{C}}MI09]{ccivril2009selecting}
Ali {\c{C}}ivril and Malik Magdon-Ismail.
\newblock On selecting a maximum volume sub-matrix of a matrix and related
  problems.
\newblock {\em Theoretical Computer Science}, 410(47-49):4801--4811, 2009.

\bibitem[DR10]{deshpander}
Amit Deshpande and Luis Rademacher.
\newblock Efficient volume sampling for row/column subset selection.
\newblock In {\em Foundations of Computer Science (FOCS), 2010 51st Annual IEEE
  Symposium on}, pages 329--338. IEEE, 2010.

\bibitem[Dys62]{Dys62}
Freeman~J. Dyson.
\newblock Statistical theory of the energy levels of complex systems. {III}.
\newblock {\em J. Mathematical Phys.}, 3:166--175, 1962.

\bibitem[GKFT14]{GilKulFox14}
Jennifer~A Gillenwater, Alex Kulesza, Emily Fox, and Ben Taskar.
\newblock Expectation-maximization for learning determinantal point processes.
\newblock In {\em NIPS}, 2014.

\bibitem[GVL12]{golub2012matrix}
Gene~H Golub and Charles~F Van~Loan.
\newblock {\em Matrix computations}, volume~3.
\newblock JHU Press, 2012.

\bibitem[Hor87]{horton1987polynomial}
Joseph~Douglas Horton.
\newblock A polynomial-time algorithm to find the shortest cycle basis of a
  graph.
\newblock {\em SIAM Journal on Computing}, 16(2):358--366, 1987.

\bibitem[KT11]{KulTas11}
Alex Kulesza and Ben Taskar.
\newblock $k$-{DPP}s: Fixed-size determinantal point processes.
\newblock In {\em Proceedings of the 28th International Conference on Machine
  Learning, {ICML} 2011, Bellevue, Washington, USA, June 28 - July 2, 2011},
  pages 1193--1200, 2011.

\bibitem[KT12]{KulTas12}
Alex Kulesza and Ben Taskar.
\newblock {\em Determinantal Point Processes for Machine Learning}.
\newblock Now Publishers Inc., Hanover, MA, USA, 2012.

\bibitem[Kul12]{Kul12}
A.~Kulesza.
\newblock {\em Learning with determinantal point processes}.
\newblock PhD thesis, University of Pennsylvania, 2012.

\bibitem[LB12]{LinBil12}
Hui Lin and Jeff~A. Bilmes.
\newblock Learning mixtures of submodular shells with application to document
  summarization.
\newblock In {\em Proceedings of the Twenty-Eighth Conference on Uncertainty in
  Artificial Intelligence, Catalina Island, CA, USA, August 14-18, 2012}, pages
  479--490, 2012.

\bibitem[LCYO16]{LeeChaYan16}
Donghoon Lee, Geonho Cha, Ming{-}Hsuan Yang, and Songhwai Oh.
\newblock Individualness and determinantal point processes for pedestrian
  detection.
\newblock In {\em Computer Vision - {ECCV} 2016 - 14th European Conference,
  Amsterdam, The Netherlands, October 11-14, 2016, Proceedings, Part {VI}},
  pages 330--346, 2016.

\bibitem[LJS16a]{li2016fast}
Chengtao Li, Stefanie Jegelka, and Suvrit Sra.
\newblock Fast dpp sampling for nystrom with application to kernel methods.
\newblock {\em International Conference on Machine Learning (ICML)}, 2016.

\bibitem[LJS16b]{LiJegSra16}
Chengtao Li, Stefanie Jegelka, and Suvrit Sra.
\newblock Fast sampling for strongly rayleigh measures with application to
  determinantal point processes.
\newblock {\em 1607.03559}, 2016.

\bibitem[Mac75]{Mac75}
Odile Macchi.
\newblock The coincidence approach to stochastic point processes.
\newblock {\em Advances in Appl. Probability}, 7:83--122, 1975.

\bibitem[MS15]{MarSra15}
Zelda Mariet and Suvrit Sra.
\newblock Fixed-point algorithms for learning determinantal point processes.
\newblock In {\em Proceedings of the 32nd International Conference on Machine
  Learning (ICML-15)}, pages 2389--2397, 2015.

\bibitem[Nik15]{nikolov2015randomized}
Aleksandar Nikolov.
\newblock Randomized rounding for the largest simplex problem.
\newblock In {\em Proceedings of the Forty-Seventh Annual ACM on Symposium on
  Theory of Computing}, pages 861--870. ACM, 2015.

\bibitem[NS16]{nikolov2016maximizing}
Aleksandar Nikolov and Mohit Singh.
\newblock Maximizing determinants under partition constraints.
\newblock In {\em STOC}, pages 192--201, 2016.

\bibitem[RK15]{rebeschini2015fast}
Patrick Rebeschini and Amin Karbasi.
\newblock Fast mixing for discrete point processes.
\newblock In {\em COLT}, pages 1480--1500, 2015.

\bibitem[RKT15]{Rising2015126}
Justin Rising, Alex Kulesza, and Ben Taskar.
\newblock An efficient algorithm for the symmetric principal minor assignment
  problem.
\newblock {\em Linear Algebra and its Applications}, 473:126 -- 144, 2015.

\bibitem[RTL76]{rose1976algorithmic}
Donald~J Rose, R~Endre Tarjan, and George~S Lueker.
\newblock Algorithmic aspects of vertex elimination on graphs.
\newblock {\em SIAM Journal on computing}, 5(2):266--283, 1976.

\bibitem[SEFM15]{summa2015largest}
Marco~Di Summa, Friedrich Eisenbrand, Yuri Faenza, and Carsten Moldenhauer.
\newblock On largest volume simplices and sub-determinants.
\newblock In {\em Proceedings of the Twenty-Sixth Annual ACM-SIAM Symposium on
  Discrete Algorithms}, pages 315--323. Society for Industrial and Applied
  Mathematics, 2015.

\bibitem[SZA13]{SnoZemAda13}
Jasper Snoek, Richard~S. Zemel, and Ryan~Prescott Adams.
\newblock A determinantal point process latent variable model for inhibition in
  neural spiking data.
\newblock In {\em Advances in Neural Information Processing Systems 26: 27th
  Annual Conference on Neural Information Processing Systems 2013. Proceedings
  of a meeting held December 5-8, 2013, Lake Tahoe, Nevada, United States.},
  pages 1932--1940, 2013.

\bibitem[Tsy09]{Tsybakov2009}
Alexandre~B. Tsybakov.
\newblock {\em Introduction to nonparametric estimation}.
\newblock Springer Series in Statistics. Springer, New York, 2009.

\bibitem[XO16]{XuOu16}
Haotian Xu and Haotian Ou.
\newblock Scalable discovery of audio fingerprint motifs in broadcast streams
  with determinantal point process based motif clustering.
\newblock {\em {IEEE/ACM} Trans. Audio, Speech {\&} Language Processing},
  24(5):978--989, 2016.

\bibitem[YFZ{\etalchar{+}}16]{YaoFanZha16}
Jin{-}ge Yao, Feifan Fan, Wayne~Xin Zhao, Xiaojun Wan, Edward~Y. Chang, and
  Jianguo Xiao.
\newblock Tweet timeline generation with determinantal point processes.
\newblock In {\em Proceedings of the Thirtieth {AAAI} Conference on Artificial
  Intelligence, February 12-17, 2016, Phoenix, Arizona, {USA.}}, pages
  3080--3086, 2016.

\end{thebibliography}
\bibliographystyle{alpha}

\end{document}